\documentclass[preprint,12pt,number,english]{elsarticle}
\usepackage{yhmath}
\usepackage{amsmath,amssymb,amsthm}
\usepackage{xcolor,soul} 
\usepackage{array,dcolumn}
\usepackage{enumitem}
\usepackage{cancel}
\usepackage[normalem]{ulem} 
\usepackage{lineno}
\usepackage{babel}
\usepackage[T1]{fontenc}
\newtheorem{thm}{Theorem}
\newtheorem{propo}{Proposition}
\newtheorem{lema}{Lemma}
\newtheorem{coro}{Corollary}
\newtheorem{defi}{Definition}
\newtheorem{rmk}{Remark}
\newtheorem{ejem}{Example}

\newcounter{letapar}

\begin{document}
\bibliographystyle{plain}
\begin{frontmatter}
\title{Combined matrices of almost strictly sign regular matrices}



\author[1]{P. Alonso} 
\ead{palonso@uniovi.es}
\author[2]{J.M. Pe\~na}
\ead{jmpena@unizar.es}
\author[1]{M.L. Serrano}
\ead{mlserrano@uniovi.es}

\address[1]{Departamento de Matem\'aticas, Universidad de Oviedo, Spain}
\address[2]{Departamento de Matem\'atica Aplicada, Universidad de Zaragoza, Spain}

\begin{abstract}
The combined matrix is a very useful concept for many applications. Almost strictly sign regular (ASSR) matrices form an important structured class of matrices   with to possible zero patterns  , which are either type-I staircase or type-II staircase. We prove that, under an irreducibility condition, the pattern of zero and nonzero entries of an ASSR matrix is preserved by the corresponding combined matrix. Without the irreducibility condition, it is proved that type-I and type-II staircase is still preserved. Illustrative numerical examples are included.
\end{abstract}

\end{frontmatter}

\section{Introduction}

Given a nonsingular real matrix $A$ of order $n$, the combined matrix, $C(A)$, is the Hadamard (entry wise) product $C(A)=A \circ (A^{-1})^T$. It is well known that each row (column) sum of $C(A)$ is equal to one (see \cite{HORN1991}). Combined matrices give the relationship between the eigenvalues and diagonal entries of a diagonalizable matrix (see \cite{HORN1991}). Furthermore, combined matrices also appear in the chemistry field (see \cite{MCAVOY1983}). If we consider the case when $C(A)$ is nonnegative, it has interesting properties and applications since it is a doubly stochastic matrix (see \cite{BRUALDI1988, MOURAD2013}).

In recent years, several studies about combined matrices have been carried out (see \cite{BRU2014,BRU2016,FIEDLER2011,FIEDLER2012}). In particular, in \cite{BRU2014,BRU2016}, the authors study some kinds of matrices such that their combined matrices are nonnegative, analyzing their relationship with the zero pattern of $A$. In this case, that is, when $C(A)$ is nonnegative, the combined matrix is doubly stochastic. Among the considered kinds of matrices, we can mention $M$-matrices, $H$-matrices, $G$-matrices and sign regular matrices.

A very important class of structured matrices due to its applications is the class of sign regular (SR) matrices. A matrix is SR if all its minors of the same order have the same sign. These matrices arise naturally in many areas of mathematics, statistics, mechanics, computer-aided geometric design, economics, etc. (see \cite{ANDO1987, PENA2003}). Totally positive matrices (matrix with all its minors nonnegative) are very important in many applications and they are trivially SR matrices. In this work we analyze the properties verified by the combined matrix associated with the almost strictly sign regular (ASSR) matrices, a subclass of the nonsingular SR matrices (see \cite{ALONSO2015}).   These matrices have all their nontrivial minors of the same order with the same strict sign and present to possible zero patterns, which are either type-I staircase or type-II staircase.  

   In this paper we prove that irreducible ASSR matrices and their combined matrices have the same pattern of zero and nonzero entries. Without the irreducibility condition, we also prove that type-I and type-II staircases are still preserved. The inheritance of the zero pattern is very important for saving computational cost when applying numerical methods. The results and examples of this paper show that the class of ASSR matrices is precisely the adequate framework for dealing with the problem of the preservation of zero pattern for combined matrices.   

This work is organized as follows. Section 2 is dedicated to general results for combined matrices. In Section 3 we present several definitions and auxiliary results for ASSR matrices. Irreducible ASSR matrices are characterized in Section 4. Combined matrices of irreducible ASSR matrices are also analyzed in Section 4, while the general case is addressed in the next section. Finally, SR matrices with SR combined matrix are studied in Section 6. Throughout the work a wide collection of illustrative examples of the proven properties is presented.

\section{General results for combined matrices}

For $k, n \in \mathbb{N}$, with $1 \leq k \leq n$, $Q_{k,n}$ denotes the set of all increasing sequences of $k$ natural numbers not greater than $n$. For $\alpha=(\alpha_{1}, \dots, \alpha_{k})$, $\beta=(\beta_{1}, \dots, \beta_{k}) \in Q_{k,n}$ and $A$ an $n \times n $ real matrix, we denote by  $A[\alpha|\beta]$ the $k \times k$ submatrix of $A$ containing rows $\alpha_{1}, \dots, \alpha_{k}$ and columns $\beta_{1}, \dots, \beta_{k}$ of $A$. If $\alpha=\beta$, we denote by $A[\alpha]:=A[\alpha|\alpha]$ the corresponding principal submatrix.

First, we present some definitions that may be of interest.

Let $A$ be a nonsingular real matrix  of order $n$ and we denote by $N$ the set of indices $\{1,2,\dots,n\}$. Given two indices $i,j \in N$, we denote by $A_{ij}$ the $(i,j)$ complementary minor of $A$, i.e., the determinant of the submatrix obtained from $A$ by removing row $i$ and column $j$.

Given a real matrix $A=(a_{ij})_{1 \leq i,j \leq n}$, we say it is nonnegative (entry wise) if $a_{ij} \geq 0$ for all $i,j \in N$ and we denote it as $A \geq 0$. If $-A \geq 0$, then we say that $A$ is nonpositive.

Let us recall that, given two matrices of order $n$, $A=(a_{ij})_{1 \leq i,j \leq n}$ and $B=(b_{ij})_{1 \leq i,j \leq n}$, the Hadamard (entry wise) product of $A$ and $B$ is defined as $A \circ B=(a_{ij} b_{ij})$, $\forall \; 1 \leq i,j \leq n$.

Taking into account the previous concepts, we can define the combined matrix in the following way:

\begin{defi}
The combined matrix of a nonsingular real matrix $A=(a_{ij})_{1 \leq i,j \leq n}$ is defined as $C(A)=A \circ (A^{-1})^T$. Then
the elements of $C(A)$ are denoted by $c_{ij}$, with $1 \leq i,j \leq n$, and can be expressed as
\begin{equation}\label{combinadapormenores}
 \quad C(A)=\left( \frac{1}{\det A}(-1)^{i+j}a_{ij}A_{ij}\right).
\end{equation}
\end{defi}

\begin{rmk}\label{aij=0->cij=0}
Note that if $C(A)=(c_{ij})_{1 \leq i,j \leq n}$ is the combined matrix of a  matrix $A$,  it follows trivially from (\ref{combinadapormenores}) that if $a_{ij}=0$ then $c_{ij}=0$.
\end{rmk}

Taking into account Lemma 8 (ii) of \cite{BRU2014}, we can conclude that the combined matrix of a triangular matrix is the identity matrix.

As usual, hereinafter we denote by $P_n$ (backward identity matrix) and by $S_n$ the matrices $n\times n$, whose element $(i, j)$ are defined as:

\begin{defi}\label{def_Pn}
Let us denote by $P_n=(p_{ij})_{1\leq i,j \leq n}$ (backward identity matrix) and as $S_n=(s_{ij})_{1\leq i,j \leq n}$,  the matrices defined by
$$
p_{ij}=\left\{\begin{array}{lr}
                               1, & \qquad \mathrm{if } \  i+j=n+1,\\
                               0, &   \mathrm{otherwise},
                               \end{array}\right., \quad
s_{ij}=\left\{\begin{array}{lr}
                               (-1)^{i-1}, & \qquad \mathrm{if } \  i=j,\\
                               0, &   \mathrm{otherwise},
                               \end{array}\right.
                               $$
that is,
$$
P_n=\left(
      \begin{array}{ccccc}
        0 & 0 & \cdots & 0 & 1 \\
        0 & 0 & \cdots & 1 & 0 \\
        \vdots &  \vdots & \vdots &  \vdots & \vdots \\
        0 & 1 & \cdots & 0 & 0 \\
        1 & 0 & \cdots & 0 & 0\\
      \end{array}
    \right), \quad
S_n=\left(
      \begin{array}{ccccc}
        1 & 0 & \cdots & 0 & 0 \\
        0 & -1 & \cdots & 0 & 0 \\
        \vdots &  \vdots & \vdots &  \vdots & \vdots \\
        0 & 0 & \cdots & (-1)^{n-2} & 0 \\
        0 & 0 & \cdots & 0 & (-1)^{n-1}\\
      \end{array}
    \right).
$$
\end{defi}

The left product by the matrix $P_n$ produces a permutation in the rows of the matrix $A$, so, the row $k$ of $P_n A$ is the row $n-k+1$ of matrix $A$, with $k= 1,\dots,n$. On the other hand, if we denote by $d_{ij}$, with $1\leq i,j \leq n$, the elements of $S_n A S_n$, it is possible to prove that
$$
d_{ij}=\left\{\begin{array}{lr}
                               a_{ij}, & \qquad \mathrm{if } \  i+j=\mathrm{even},\\
                               -a_{ij}, & \qquad \mathrm{if } \  i+j=\mathrm{odd}.
                               \end{array}\right.
$$

Taking into account Lemma 5.4.2 (b) of \cite{HORN1991}, we present the following result:

\begin{lema}\label{Lema-HORN}
Let $A=(a_{ij})_{1 \leq i,j \leq n}$ be a $n \times n$ nonsingular matrix. Then the following properties are satisfied:
\begin{enumerate}
\item[a)] If $D$ and $E$ are nonsingular diagonal matrices, then $C(D A E)=C(A)$. In particular $C(A)=C(-A)=C(S_n A)=C(A S_n)=C(S_n A S_n)$.
\item[b)] If $P$ and $Q$ are permutation matrices, then $C(P A Q)=P C(A) Q$. In particular $C(P_n A)=P_n C(A)$.
\end{enumerate}
\end{lema}

It is well known that each row (column) sum of $C(A)$ is equal to one (see Lema 5.4.2 (a) of \cite{HORN1991}). Considering this result, it is immediate to affirm that:

\begin{lema}\label{nonoposi}
Let $A$ be a nonsingular matrix. Then $C(A)$ is not nonpositive.
\end{lema}

\section{Definitions and auxiliary results for ASSR matrices}

Taking into account the importance of the signs of minors, we define the vector of signatures, which is where these signs are stored.

\begin{defi}\label{defsig}
A vector $\varepsilon=(\varepsilon_1, \varepsilon_2, \dots , \varepsilon_ n)\in \mathbb{R}^n$ is a signature sequence, or simply,
a signature,   if   $\varepsilon_i\in \left\{+1, -1 \right\}$, $\forall i \in \mathbb{N}$, $ 1 \leq i \leq n$.
\end{defi}

Next, we define SR and SSR matrix.

\begin{defi}
A real $n\times n$ matrix  $A$ is said to be SR, with signature $\varepsilon=(\varepsilon_1, \varepsilon_2, \dots , \varepsilon_n)$, if all its minors  satisfy that
\begin{equation}\label{def_sr_gen}
  \varepsilon_m\det A[\alpha | \beta ] \geq 0, \quad \alpha , \beta \in Q_{m,n}, \qquad m \leq n.
\end{equation}
When the sign of minors is strict, we will say that the matrix is strictly sign regular (SSR).
\end{defi}

Note that $\varepsilon_i=+1$ ($\varepsilon_i=-1$) $\forall \; 1 \leq i \leq n$ implies that all the minors of order $i$ are positive (negative) or zero. Considering a nonsingular matrix,  all the minors of order $i$ cannot be zero.

It should be noted that if $A$ is a nonsingular SR matrix with signature $\varepsilon=(\varepsilon_1, \varepsilon_2, \dots , \varepsilon_n)$, then (see Theorem 3.3 of \cite{ANDO1987}) the signature of the matrix $S_n A^{-1}S_n$ is obtained from the signature of matrix $A$ as
$$\varepsilon_i(S_n A^{-1}S_n)=\varepsilon_n \varepsilon_{n-i},$$
with   $\varepsilon_0:=1$  . 

From (\ref{combinadapormenores}), it is convenient to observe the relevance of the signs of minors of order $1$, $n-1$ and $n$ ($\varepsilon_1$, $\varepsilon_{n-1}$ and $\varepsilon_n$) in the calculation of the combined matrix for the different types of SR matrices.

Let $A=(a_{ij})_{1 \leq i,j \leq n}$ be a matrix of order $n$. It is said that $A$ has a checkerboard pattern if $sign(a_{ij})=(-1)^{i+j}$ or $a_{ij}=0$ for all $i,j \in N$. The   proof   of the following result is straightforward.

\begin{lema}\label{nulos-patron}
Let $A=(a_{ij})_{1 \leq i,j \leq n}$ be a nonsingular SR matrix with signature $\varepsilon=(\varepsilon_1, \varepsilon_2, \dots , \varepsilon_n)$ and $C(A)=(c_{ij})_{1 \leq i,j \leq n}$ their combined matrix. Then either $C(A)$ or $-C(A)$ has a checkerboard pattern.
\end{lema}

Many of the involved matrices in the results that will be analyzed are structured matrices,   in the sense that   present their zero and nonzero entries grouped in certain positions; for this reason, the type-I and type-II staircase matrices are defined below (see Lema 7 de \cite{HUANG2012}).

\begin{defi}\label{def_Iescalonadas}
A real matrix $A=\left(a_{ij}\right)_{1 \leq i,j\leq n}$ is called \textit{type-I staircase} if it satisfies simultaneously the following conditions
\begin{eqnarray}
   a_{11} \not= 0, \ a_{22}\not= 0, \dots ,  a_{nn}\not= 0; \label{diagnonulatipoI}\\
   a_{ij}=0, i>j \Rightarrow a_{kl}=0, \ \forall \, l \leq j, \ i \leq k;\label{cerosdebajotipoI}\\
   a_{ij}=0, i<j \Rightarrow  a_{kl}=0, \ \forall \, k \leq i, \  j \leq l. \label{cerosarribatipoI}
\end{eqnarray}
\end{defi}

\begin{defi}\label{def_IIescalonadas}
A real matrix $A=\left(a_{ij}\right)_{1 \leq i,j\leq n}$ is called \textit{type-II staircase} if it satisfies simultaneously the following conditions
\begin{eqnarray}
a_{1n} \not= 0, \ a_{2,n-1}\not= 0, \  \dots , a_{n1}\not= 0;\label{diagnonulatipoII}\\
a_{ij}=0, j>n-i+1 \Rightarrow a_{kl}=0,\ \forall \, i \leq k, \; j \leq l;\label{cerosdebajotipoII}\\
a_{ij}=0, j<n-i+1 \Rightarrow a_{kl}=0,\ \forall \, k \leq i, \; l \leq j. \label{cerosarribatipoII}
\end{eqnarray}
\end{defi}

Next we define the nontrivial submatrices, which are described from the pattern of zeros of a staircase matrix (type-I or type-II).

\begin{defi}\label{sub-no-trivial-tipo-I}
For a real matrix $A=\left(a_{ij}\right)_{1 \leq i,j\leq n}$ type-I (type-II) staircase, a submatrix $A[\alpha | \beta ]$, with $\alpha, \beta \in Q_{m,n}$, is nontrivial if all its main diagonal (backward diagonal) entries are nonzero, that is, if $a_{\alpha_i,\beta_i}\not=0$ ($a_{\alpha_i,\beta_{m-i+1}}\not=0$), $\forall \; i=1,\dots , m$. The minor associated to a nontrivial submatrix ($A[\alpha | \beta]$) is called a nontrivial minor ($\det A [\alpha | \beta ]$).
\end{defi}

Finally, the ASSR matrices are introduced, which will be studied in more depth in this work.

\begin{defi}
An $n\times n$ real matrix  $A$ is said to be ASSR with signature $\varepsilon=(\varepsilon_1, \varepsilon_2, \dots , \varepsilon_ n)$ if it is either type-I or type-II staircase and all its nontrivial minors $\det A[\alpha | \beta ]$ satisfy that
\begin{equation}\label{def_assr_gen}
  \varepsilon_m\det A[\alpha | \beta ]>0, \quad \alpha , \beta \in Q_{m,n}, \qquad m \leq n.
\end{equation}
\end{defi}

Note that if $A$ is ASSR, then it is nonsingular (see Theorem 1.32 de \cite{SERRANO2017}).

The relationship between the type-I and type-II staircase matrices is obtained from the backward identity matrix:

\begin{rmk}\label{tipoIPntipoII}
Let $A=(a_{ij})_{1 \leq i,j \leq n}$ be an $n\times n$ real matrix and $P_n$ be the backward identity matrix. Then:

\begin{tabular}{c}
if $A$ is type-I staircase $\Rightarrow$ $P_nA$ is type-II  staircase,\\
if $A$ is type-II staircase $\Rightarrow$ $P_nA$ is type-I staircase.\\
\end{tabular}
\end{rmk}

Taking into account that $P_nA$ is a matrix whose rows are the same that those of $A$, but reordered, it is possible to obtain the signature of an ASSR matrix $A$ from the signature of $P_nA$ (see Corollary 1 de \cite{ALONSO2015}).

\begin{propo}\label{sigPnA}
A matrix $A=(a_{ij})_{1 \leq i,j \leq n}$ is ASSR if and only if $P_nA$ is also ASSR. Furthermore, if the signature of $A$ is $\varepsilon=(\varepsilon_1, \varepsilon_2, \dots , \varepsilon_n)$, then the signature of $P_nA$ is $\varepsilon'=(\varepsilon'_1, \varepsilon'_2,\dots , \varepsilon'_n)$, with
$\varepsilon'_m=(-1)^{\frac{m(m-1)}{2}}\varepsilon_m$
\end{propo}

\section{Combined matrices of irreducible ASSR matrices}

A square matrix $A$ is called reducible if there exists  a permutation matrix $P$ such that $PAP^T$ is block triangular:
$$
PAP^T=\left(\begin{array}{cc}
B&0\\
C&D
\end{array}\right),
$$
where $B$ and $D$ are square matrices. Otherwise, $A$ is said to be irreducible. Clearly, an $n \times n$ matrix $A$ is reducible if and only if there exists an ordering $(i_1, i_2,\ldots ,i_k,j_{k+1},$ $j_{k+2}, \ldots, j_n)$ of $(1,2,\ldots,n)$ such that $A[i_1, i_2,\ldots ,i_k|j_{k+1},$ $j_{k+2}, \ldots, j_n]=0$.

Let us recall that a totally   positive   square matrix is oscillatory if there exists a positive integer $k$ such that $A^k$ is   stricly   totally positive. It is well-known that an oscillatory matrix is irreducible. We now characterize the irreducible ASSR matrices of type-I staircase.

\begin{thm} \label{A}
Let $A$ be a type-I staircase ASSR matrix. Then $A$ is irreducible if and only if  its subdiagonal and superdiagonal have no zero entries, that is, $a_{j+1,j}\ne 0$ and $a_{j,j+1}\ne 0$ for all $j=1,\dots,n-1$.
\end{thm}

\begin{proof}
First, let us assume that $A$ has a zero subdiagonal entry, i.e., $a_{i,i-1}=0$ for some $i>0$. Then $A[i,i+1,\dots,n |1,\dots,i-1]=0$ and so $A$ is reducible. If $A$ has a zero superdiagonal entry, i.e., $a_{i,i+1}=0$ for some $i<n$, then $A[1,\dots,i |i+1,\dots,n]=0$ and so $A$ is again reducible.

For the converse, let us suppose that $A$ is reducible. So, it has a null submatrix of the form $A[i_1, i_2,\ldots ,i_k|j_{k+1},$ $j_{k+2}, \ldots, j_n]$, where $(i_1, i_2,\ldots ,i_k,j_{k+1},$ $j_{k+2}, \ldots, j_n)$ is a permutation of $(1,2,\ldots,n)$. Let us first assume that $i_k<n$. Then $i_{k}+1\in \{ j_{k+1}, j_{k+2}, \ldots, j_n\}$ and so $a_{i_k,i_{k}+1}=0$. Finally, let us assume that $i_k=n$. This implies that $j_n<n$. Then $j_n+1\in \{ i_1, i_2,\ldots ,i_k\}$ and so $a_{j_n+1,j_n}=0$.
\end{proof}

Regarding the zero pattern, next, we prove that, under the condition of irreducibility, the zero pattern is preserved.

\begin{thm}\label{a=0<->c=0}
Let $A=(a_{ij})_{1 \leq i,j \leq n}$ be an irreducible type-I staircase  ASSR matrix with signature $\varepsilon=(\varepsilon_1, \varepsilon_2,\dots , \varepsilon_n)$. Let $C(A)=(c_{ij})_{1 \leq i,j \leq n}$ be its combined matrix. Then
$$
a_{ij}\not=0 \Leftrightarrow  c_{ij}\not= 0,
$$
$1\le i,j\le n$.
\end{thm}

\begin{proof}
Let us suppose that  $a_{ij}\not= 0$ with $i>j$. By (\ref{combinadapormenores})
$$
c_{ij}=\dfrac{1}{\det(A)}(-1)^{i+j}a_{ij}A_{ij}.
$$
To prove that $c_{ij}\not=0$ it is sufficient to show that $A_{ij}\not= 0$. Given that $A$ is an ASSR matrix, we prove that the submatrix $M$ associated to $A_{ij}$ is a nontrivial submatrix. This submatrix is
$$
A[1,\dots,i-1,i+1,\dots,n|1,\dots,j-1,j+1,\dots,n ]=
$$
$$=\left(
                                                      \begin{array}{c|c|c}
                                                        A[1,\dots,j-1] & * & * \\ \hline
                                                         *  & A[j,\dots,i-1|j+1,\dots,i] & * \\ \hline
                                                         * &  * &  A[i+1,\dots,n]\\
                                                      \end{array}
                                                    \right).
$$

The diagonal entries of the main diagonal of $M$ are $a_{11}, a_{22}, \dots a_{j-1,j-1}$ (diagonal entries of $A[1,\dots,j-1] $), $a_{i+1,i+1}, a_{i+2,i+2},\dots, a_{nn}$ (diagonal entries of $ A[i+1,\dots,n]$) and $a_{j,j+1},a_{j+1,j+2},\dots , a_{i-1,i}$ (diagonal entries of $A[j,\dots,i-1|j+1,\dots,i]$). The elements $a_{11}, a_{22}, \dots a_{j-1,j-1}$ and $a_{i+1,i+1}, a_{i+2,i+2},\dots, a_{nn}$ are in the main diagonal of $A$, and therefore they are nonzero. On the other hand, as $A$ is an irreducible type-I staircase ASSR matrix, by Theorem \ref{A} the elements $a_{j,j+1},a_{j+1,j+2},\dots , a_{i-1,i}$ are nonzero. Then $A_{ij}$ is associated to a nontrivial submatrix of $A$ and it is nonzero. In consequence, $c_{ij}\not= 0$.

Analogously, if $i<j$ we can consider the transpose of $M$ to deduce the result.

For the converse, it is enough to take into account   equation   (\ref{combinadapormenores}).
\end{proof}

\begin{rmk}\label{eje-SR}
It is worth noting that a nonsingular irreducible SR  matrix does not necessarily satisfy the conclusion of the previous result, as the following example shows. Let $A_1$ be an irreducible SR nonsingular matrix with signature $\varepsilon=(-1,+1,-1)$
$$
A_1=
\left(\begin{array}{ccc}
 -1 & -3 &  -5 \\
 -1 & -6 &  -10 \\
 -1 & -15 & -29
 \end{array}\right).
$$
The combined matrix  ($C(A_1)=A_1 \circ (A_1^{-1})^T$) is
$$
C(A_1)=
\left(\begin{array}{ccc}
 2 & -19/4 &  15/4 \\
 -1 & 12 &  -10 \\
 0 &  -25/4 & 29/4
 \end{array}\right).
$$
Notice that the element $(3,1)$ of $C(A_1)$ is zero, while that in this position $A_1$ has the value $-1$.
\end{rmk}

From Lemma \ref{Lema-HORN},   Remark   \ref{tipoIPntipoII} and Theorem \ref{A}, we can deduce the following consequence.

\begin{coro}\label{Coro-A}
If $A=(a_{ij})_{1 \leq i,j \leq n}$ is an ASSR matrix, $C(A)=(c_{ij})_{1 \leq i,j \leq n}$ its combined matrix  and $A$ is either irreducible type-I staircase or type-II staircase with $P_nA$ irreducible, then $a_{ij}\ne 0$ if and only if $c_{ij}\ne 0$.
\end{coro}

To conclude this section, we present an example that illustrates some of the properties of the combined matrices associated to an irreducible ASSR matrix.

\begin{ejem} \label{e}
Let $A_2$ be an ASSR matrix with signature $\varepsilon=(-1,+1,-1,+1,$ $-1,-1)$
$$
A_2=
\left(\begin{array}{cccccc}
 -1 &  -2 &   0 &   0 &   0  &   0\\
 -4 & -10 &  -6 &  -8 &   0  &   0\\
  0 & -10 & -33 & -46 &  -9  &   -6\\
  0 & -16 & -60 & -92 & -60  & -36\\
  0 &  -2 & -21 & -70 & -242 & -443\\
  0 &   0 &   0 & -36 & -316 & -2823
 \end{array}\right).
$$
If the combined matrix of $A_2$ is calculated, we obtain that   $C(A_2)=A_2 \circ (A_2^{-1})^T$    is
$$
1.0e+06 *
\left(\begin{array}{cccccc}
 -0.6709 &  0.6709 &  0      &  0      &  0      &  0\\
  0.6709 & -0.8386 &  1.6695 & -1.5018 &  0      &  0\\
  0      &  0.3173 & -3.4744 &  3.2674 & -0.1134 &  0.0030\\
  0      & -0.1521 &  1.8926 & -1.9579 &  0.2264 & -0.0091\\
  0      &  0.0025 & -0.0877 &  0.1973 & -0.1210 &  0.0089\\
  0      &  0      &  0      & -0.0051 &  0.0079 & -0.0028
 \end{array}\right).
$$
Considering the obtained combined matrix  it is observed that the matrix $C(A_2)$ is not ASSR, since all the elements of $C(A_2)$ do not have the same sign. Nevertheless, as $A_2$ is irreducible, the zero pattern of $A_2$ is preserved (Theorem \ref{a=0<->c=0}).
As expected,  the sum of the elements of the rows (columns) of the combined matrix is equal to one, but this matrix is not doubly stochastic.
Besides, the matrix $C(A_2)$ does not have a checkerboard pattern, because $sign(c_{ij})=-(-1)^{i+j}$ or $c_{ij}=0$, nevertheless we can assure that $-C(A_2)$ has a checkerboard pattern (Lemma \ref{nulos-patron}).
\end{ejem}

\section{Combined matrices of general ASSR matrices}

Under the conditions of Theorem \ref{a=0<->c=0}, $C(A)$ is type-I staircase with the same zero pattern as $A$. In this section we analyze some  properties of the combined matrices associated to ASSR matrices when $A$ is not a irreducible matrix.
We shall prove that, if $A$ is type-I staircase, then $C(A)$ is also type-I staircase. However, the zero pattern is not preserved, as shown in the following example.

\begin{ejem}
Let $A_3$ be an ASSR matrix with signature $\varepsilon=(-1,+1,-1)$
$$
A_3=
\left(\begin{array}{ccc}
 -1 & -2 &  0 \\
 -1 & -3 &  0 \\
 -1 &  -4 & -5
 \end{array}\right).
$$
The combined matrix associated $C(A_3)=A_3 \circ (A_3^{-1})^T$ turns out to be
$$
C(A_3)=
\left(\begin{array}{ccc}
 3 & -2 &  0 \\
 -2 & 3 &  0 \\
 0 &  0 & 1
 \end{array}\right),
$$
so, it is possible to conclude that the combined matrix does not maintain the zero pattern of the initial matrix in all cases.
\end{ejem}

The next result shows that, when an ASSR matrix is considered, then the pattern type-I staircase is preserved by the combined matrix.

\begin{thm}\label{CA-tipo-I}
Let $A=(a_{ij})_{1 \leq i,j \leq n}$ be an ASSR matrix with signature $\varepsilon=(\varepsilon_1,\varepsilon_2,\dots , \varepsilon_n)$. Let $C(A)=(c_{ij})_{1 \leq i,j \leq n}$ be the combined matrix associated to $A$. If $A$ is type-I staircase, then $C(A)$ is also type-I staircase.
\end{thm}

\begin{proof}
We will check that $C(A)$ verifies the conditions of Definition \ref{def_Iescalonadas}.

The  diagonal entries of   $C(A)$   verify that
$$
c_{ii}=\dfrac{1}{\det A}a_{ii}A_{ii}.
$$
If row and column $i$ of $A$ are removed, then the nontrivial submatrix $M=A[1,\dots,i-1,i+1,\dots,n]$ is obtained and $A_{ii}=\det M$. Taking into account this condition, it is possible to conclude that $A_{ii}$ is nonzero.

Let us consider now that $i> j$ and that $c_{ij}=0$, then
\begin{equation}\label{cij=0}
c_{ij}=\dfrac{1}{\det(A)}(-1)^{i+j}a_{ij}A_{ij}=0.
\end{equation}
We are going to prove that $c_{i+1,j}=0$ and that $c_{i,j-1}=0$, these conditions are sufficient to fulfill (\ref{cerosdebajotipoI}).

As $A$ is type-I staircase, if $a_{ij}=0$, then $a_{i+1,j}=0$ and so $c_{i+1,j}=0$.

Let us now suppose  that $a_{ij} \neq 0$. Then as $c_{ij}=0$, by (\ref{cij=0}), it follows that $A_{ij}=0$. Analogous to the reasoning made in the proof of Theorem \ref{a=0<->c=0}, we can prove that some of the elements $a_{j,j+1},a_{j+1,j+2},\dots , a_{i-1,i}$ must be null, and so
\begin{equation}\label{prod}
a_{j,j+1} a_{j+1,j+2} \cdots  a_{i-1,i}=0.
\end{equation}

In order to prove that $c_{i+1,j}=0$, we analyze $A_{i+1,j}$. The minor $A_{i+1,j}$ derives from the submatrix
$$
H=
A[1,\dots,i,i+2,\dots,n|1,\dots,j-1,j+1,\dots,n ]=
$$
$$=\left(
                                                      \begin{array}{c|c|c}
                                                        A[1,\dots,j-1] & * & * \\ \hline
                                                         *  & A[j,\dots,i|j+1,\dots,i+1] & * \\ \hline
                                                         * &  * &  A[i+2,\dots,n]\\
                                                      \end{array}
                                                    \right).
$$
Analogously as was demonstrated in the proof of Theorem \ref{a=0<->c=0}, the entries that determine if $A_{i+1,j}$ is zero or not, are the  diagonal entries of $A[j,\dots,i|j+1,\dots,i+1]$, which are $a_{j,j+1},a_{j+1,j+2},\dots , a_{i-1,i}, a_{i,i+1}$. By (\ref{prod}), $ a_{j,j+1}\cdots a_{i-1,i}$ $a_{i,i+1}=0$ and then, this minor is  null.

In a similar way it can be proved that $c_{i,j-1}=0$.

Analogously, working with the matrix $A^T$, it is shown that the elements that are above the main diagonal verify the condition (\ref{cerosarribatipoI}).

So, the matrix $C(A)$ is type-I staircase.
\end{proof}

\begin{rmk} \label{r}
The previous result is not necessarily satisfied when a nonsingular type-I staircase SR matrix is considered. In order to observe this fact the next example is presented. Let $A_4$ be a nonsingular type-I staircase SR matrix with signature $\varepsilon=(+1,+1,-1)$
$$
A_4=\left(\begin{array}{ccc}
1 & 2 & 0\\
2 & 4 & 3\\
2 & 5 & 8
\end{array}\right).
$$
The combined matrix associated can be expressed as
$$
C(A_4)=A_4 \circ (A_4^{-1})^T=
\left(\begin{array}{ccc}
-5.6667 & 6.6667 & 0 \\
 10.667 &-10.667 &  1.0000\\
-4.0000 & 5.0000 & 0
\end{array}\right).
$$
Note that the entry $(3,3)$ of $C(A_4)$ is zero and hence this matrix is not type-I staircase.
\end{rmk}

The following result proves that the combined matrices also preserve the property of being type-II staircase.

\begin{coro}\label{Coro-B}
Let $A=(a_{ij})_{1 \leq i,j \leq n}$ be an ASSR matrix wiht signature $\varepsilon=(\varepsilon_1,\varepsilon_2,\dots , \varepsilon_n)$. Let $C(A)=(c_{ij})_{1 \leq i,j \leq n}$ be the combined matrix associated to $A$. If $A$ is type-II staircase, then $C(A)$ is also type-II staircase.
\end{coro}

\begin{proof}
It is sufficient to note that if $A$ is type-II staircase, then $B=P_nA$ is type-I staircase, so by Theorem \ref{CA-tipo-I} $C(B)$ is type-I staircase and  by Lemma \ref{Lema-HORN}, the combined matrix of $A$, $C(A)=C(P_n B)=P_nC(B)$, is type-II staircase.
\end{proof}

\section{SR matrices with SR combined matrix and related examples}

Note that, if $A \geq 0$ ($A \leq 0$), then $C(A)\geq 0$ if and only if $A^{-1} \geq 0$ ($A^{-1} \leq 0$) (see Theorem 13 of \cite{BRU2014}). Matrices that satisfy this property are necessarily nonnegative monomial matrices, as are the corresponding combined matrices. A matrix $A$ is said to be monomial if each row and column has exactly one nonzero entry. Monomial matrices are indeed permutationally similar to diagonal matrices   with orthogonal rows (see \cite{Garcia-Planas2015})  .

In the results of \cite{BRU2016}, for a SR   matrix   $A$, it is required as an hypothesis that $C(A)\geq 0$ or it is proved that $C(A)$ is not nonnegative. In the following result we characterize the SR matrices whose combined matrix is SR.

\begin{thm}
Let $A$ be a nonsingular SR matrix. Then, the following conditions are equivalent:
\begin{enumerate}
\item[a)] $C(A)$ is SR.
\item[b)] $C(A)\geq 0$.
\item[c)] Either $A=D_n$, with $D_n=(d_{ij})_{1 \leq i,j \leq n}$ such that $d_{ij}=0$ for all $i \neq j$ and $d_{ii} d_{jj} > 0$, $\forall i \neq j$, or $A=P_n D_n$.
\item[d)] Either $C(A)=I_n$ or $C(A)=P_n$.
\end{enumerate}
\end{thm}

\begin{proof}
a) $\Rightarrow$ b) Since $C(A)$ is SR, either $C(A) \geq 0$ or $C(A) \leq 0$. Then, by Lemma \ref{nonoposi}, b) holds.

b) $\Rightarrow$ c)   If the matrix $A$ is type-I staircase and $\varepsilon_1=+1$ ($\varepsilon_1=-1$), then $A \geq 0$ ($A \leq 0$), and $C(A) \geq 0$ if and only if $A^{-1} \geq 0$ ($A^{-1} \leq 0$) (see Theorem 13 de \cite{BRU2014}). This occurs if $A$ is a monomial matrix nonnegative (nonpositive), that is, $A=D_n=diag(d_{11}, d_{22}, \dots, d_{nn})$ with $d_{ii} d_{jj} > 0$, $\forall i \neq j$.  

c) $\Rightarrow$ d) If $A=D_n$, with $D_n=(d_{ij})_{1 \leq i,j \leq n}$ such that $d_{ij}=0$ for all $i \neq j$ and $d_{ii} d_{jj} > 0$, $\forall i \neq j$, it is immediate that $C(A)=I_n$. In the case that $A=P_n D_n$, it is concluded that $C(A)=P_n$.

d) $\Rightarrow$ a) Both matrices, $I$ and $P_n$ are SR.
\end{proof}

The previous result shows that requiring a nonsingular SR matrix the property that its combined   matrix   preserves the sign regularity is very restrictive, since this situation occurs only in the case of nonnegative combined and only very special monomial matrices can appear.

The last question that we address is whether, given a nonsingular SR matrix $A$, the matrix $|C(A)|$  preserves the sign regularity, and the analogous question with the property of being ASSR. Recall that with $|C(A)|$ we represent the matrix formed by the absolute values of the elements of $C(A)$.

The numerical experimentation carried out has shown that a general conclusion can not be established, since we have found numerous examples where $|C(A)|$ preserves the properties of being nonsingular SR or ASSR, and also a great number of them where those properties are not inherited.

We start by presenting an illustrative example in which $|C(A)|$ inherits the indicated properties.

\begin{ejem}
Let $A_5$ be an ASSR matrix (and hence is nonsingular SR) with signature $\varepsilon=(-1,-1,+1)$
$$
A_5=\left(\begin{array}{ccc}
-0.00001 & -1 & -1\\
-2 & -5 & -2\\
-3 & -1 & 0
\end{array}\right).
$$
In this case, the combined matrix obtained is
$$
C(A_5)=\left(\begin{array}{ccc}
0.0000028  & -0.8571 & 1.8571 \\
-0.2857 &  2.1429 &  -0.8571\\
 1.2857 & -0.2857 & 0
\end{array}\right).
$$

Observing the combined matrix, we have that $C(A_5)$ is not ASSR, although the  type-II and zero pattern of $A_5$ is preserved (corollaries \ref{Coro-A} y \ref{Coro-B}). On the other hand, we can verify that $|C(A_5)|=S_4 C(A_5) S_4$ is ASSR, particularly nonsingular SR. In addition, the matrix $C(A_5)$ has a checkerboard pattern (Lemma \ref{nulos-patron}).
\end{ejem}

Finally, we include an example for which $|C(A)|$ does not inherit the properties.

\begin{ejem}
Let $A_6$ be an ASSR matrix (nonsingular SR) with signature $\varepsilon=(-1,+1,+1,-1)$
$$
A_6=
\left(\begin{array}{cccc}
 -260  & -100 &  -71  &   0 \\
 -179 & -70 &  -51 & -10 \\
 -10 & -4  & -3  & -1 \\
  0 &  -1 &   -1 & -1
 \end{array}\right).
$$
The matrix $C(A_6)=A_6 \circ (A_6^{-1})^T$ is
$$
\left(\begin{array}{cccc}
 14.7170  & -98.1132 &  84.3962   &   0 \\
 -43.9057 & 250.9434 &  -211.6981 & 5.6604 \\
 30.1887  & -154.6415  & 130.1887  & -4.7358 \\
  0 &  2.8113 & -1.8868  & 0.0755
 \end{array}\right).
$$

The matrix $C(A_6)$ is not ASSR, although $C(A_6)$ is type-I staircase and the zero pattern of $A_6$ is maintained (theorems \ref{a=0<->c=0} y \ref{CA-tipo-I}). The matrix $|C(A_6)|$ is not nonsingular SR, since, for instance, the minors of order two $|C(A_6)|[1:2,2:3]$ and $|C(A_6)|[1:2,3:4]$ have opposite signs. Finally, the matrix $C(A_6)$ has a checkerboard pattern (Lemma \ref{nulos-patron}).
\end{ejem}

\section*{Acknowledgements}
This work has been partially supported by the Spanish Research   Grants   MTM2015-65433-P, TIN2017-87600-P and MTM2017-90682-REDT.


\end{document}